\documentclass[12pt,a4paper]{article}

\usepackage{amsmath,amssymb,amsfonts,amsthm} 
\usepackage{graphics,graphicx}                               
\usepackage{hyperref,fancyhdr}                 
\usepackage[english]{babel}
\usepackage{mathrsfs}
\usepackage[usenames,dvipsnames]{color}
\usepackage[final]{showkeys}
\usepackage[center]{subfigure}
\usepackage{times}

\newcommand{\R}{\ensuremath\mathbb{R}}
\newcommand{\Z}{\ensuremath\mathbb{Z}}

\newcommand{\N}{\ensuremath\mathbb{N}}

\renewcommand{\P}{\ensuremath\mathbb{P}}

\newcommand{\supp}{\operatorname{supp}}
\newcommand{\spn}{\operatorname{span}}

\newcommand{\SplSp}{\mathscr{S}_{n}(T,\Omega)}
\newcommand{\tSplSp}{\mathscr{S}_{n}^*(T,\Omega)}

\newcommand{\calC}{\ensuremath\mathcal{C}}

\newcommand{\calB}{\ensuremath\mathcal{B}}
\newcommand{\calN}{\ensuremath\mathcal{N}}

\theoremstyle{plain}%
\newtheorem{lemma}{Lemma}
\newtheorem{theorem}{Theorem}
\newtheorem{definition}{Definition}

\begin{document}

\title{Anisotropic spline approximation with non-uniform B-splines}

\author{N. Sissouno$^{\ast}$\thanks{$^\ast$Corresponding author. Email:
sissouno@forwiss.uni-passau.de
\vspace{6pt}}
}
\date{September 5, 2016}

\maketitle

\begin{abstract}
Recently the author and U. Reif introduced the concept of
diversification of uniform tensor product B-splines.
Based on this concept, we give a new constructive modification of
\mbox{non-uni}form B-splines. The resulting spline spaces are perfectly
fitted for the approximation of functions
defined on domains \mbox{$\Omega\subset \R^2$}. We build a bounded
quasi-interpolant and prove that for our spline spaces
an anisotropic error estimate in the $L^p$-norm, $1\le p\le\infty$,
is valid. In particular, we show that the constant of the error
estimate does not depend on the shape of $\Omega$ or the knot grid.

\end{abstract}


\section{Introduction}
B-splines are used in different fields of applications such as
computer aided design and finite element analysis.
In this context, lately, tensor product splines became important
due to their good approximation properties over $\R^d$ or boxes and
their efficient implementation \cite{Hughes,Hoellig2}.
However, over arbitrary domains tensor product spline approximation
cannot be used directly since some of those properties get lost.
The concept of diversification of uniform tensor product B-splines
over subsets $\Omega\subset\R^d$ is crucial for the construction
\cite{ReS} of a spline space such that an anisotropic
error estimate with respect to the sup-norm can be obtained.
Diversification itself is not connected to the uniformity of
tensor product B-splines but the methods for proving the error
estimate are tied to it. The aim of this paper is a generalisation
of the results in \cite{ReS} in two different ways. We adapt the
construction to arbitrary knot sequences and, at the same time,
consider the error with respect to the $L^p$-norm for
\mbox{$1\le p \le\infty$.}

For the spline space $\SplSp$ of standard tensor
product splines of coordinate order ${n = (n_1,\dots,n_d)}$ with knots
$T = (T_1,\dots,T_d)$,
$T_\sigma = (\dots\le\tau_{\sigma,-1}\le\tau_{\sigma,0}
	\le\tau_{\sigma,1}\le\dots)$
with $\tau_{\sigma,k}<\tau_{\sigma,k+n_\sigma}$,
$\sigma\in\{1,\dots,d\}$, restricted to $\Omega$ \cite{DeVore} show
that 
\begin{equation}
\label{eq:stdest}
  \min_{s \in \SplSp} \|f-s\|_{\Omega,L^p}
	\le C \sum_{i=1}^d h_i^{n_i} \|\partial_i^{n_i} f\|_{\Omega,L^p}
\end{equation}
for some $C >0$, where $h = (h_1,\dots,h_d)$ with
$h_\sigma:=\max_k|\tau_{\sigma,k}-\tau_{\sigma,k+1}|$ is the maximal
spacing of knots. As already stated in \cite{ReS}, the problem is
that, on one hand, the set of appropriate domains is strongly
restricted and, on the other hand, the constant $C$ is depending on the
aspect ratio $\varrho:=\max_{i,j}h_i/h_j$. The latter problem appears in a
similar fashion in \cite{HoR} and \cite{MoessnerReif}.

It is hardly surprising that the spline space $\SplSp$ is not the
best choice for approximation on arbitrary subsets of $\R^d$.
So, we construct a larger space $\tSplSp$ of non-uniform diversified
splines and prove that estimate \eqref{eq:stdest} is valid for bivariate
approximations. It holds true for the same broad class of domains
$\Omega\in\R^2$ as in the uniform case and the constant $C$ is
independent of the aspect ratio.

The main differences between the uniform case \cite{ReS} and the
non-uniform case considered here are parts of the proof of the main
result. Particularly, the concept of condensation and the construction
of a quasi-interpolant as well as the proof of its boundedness are
bound to the uniformity of the knots.
We develop a new non-uniform condensation and give a construction
of a quasi-interpolant for the resulting B-splines which is bounded
in the $L^p$-norm for $1\le p\le \infty$.

We present the univariate non-uniform condensation and some auxiliary
univariate results in the second section. 

In the third section, we introduce our main theorem claiming that
\eqref{eq:stdest} is true for approximation with non-uniform
diversified B-splines. The section also contains necessary notions
and the definition of the set of domains for which the theorem holds.

The main ingredients of the proof of the main theorem are
non-uniform condensation, a local knot structure, and a
quasi-interpolant. All those ingredients and the final proof are
stated in the fourth section.

Although the key elements for the proof are new, the basic structure
as well as diversification and the set of domains are the same as in
the uniform case. Therefore, we will
closely follow the structure and notation of \cite{ReS}. For
convenience of the reader, we will recall relevant material from
\cite{ReS} briefly, thus making our exposition self-contained.
\section{Issues in one variable}\label{123}
In this section, the knot sequence is given by $T=\{\tau_k\}_{k\in\Z}$
with $\tau_k\le\tau_{k+1}$ and $\tau_{k}<\tau_{k+n}$. As in the
uniform case, this sequence is not a priori well adapted to
a given domain
$\omega = (\underline\omega, \overline{\omega})$ in the sense that
$|\omega| := \overline{\omega}-\underline{\omega}$ may be much smaller
than the maximal knot spacing $h$. To adapt the knot sequence to
the domain, two cases need to be considered:
\begin{enumerate}
\item If $|\omega| \le h$, then the knot sequence is changed by
	replacing all knots outside of the domain by	multiple knots at the
	boundary.
	The new knot sequence is given by
	$T^\omega:=\{\tau^\omega_k\}_{k\in \Z}$, where
	\[
	\tau^\omega_k:=
	\begin{cases}
	\underline\omega &
		\text{for }k\in\{j\in\Z: \tau_j\le \underline{\omega}\} ,\\
	\tau_k & \text{for }k\in\{j\in\Z: \tau_j\in \omega\} ,\\
	\overline{\omega} &
		\text{for }k\in\{j\in\Z: \tau_j\ge \overline{\omega}\}.
	\end{cases}
	\]
\item If $|\omega| > h$, then nothing is changed, i.e.,
	$T^\omega := T$.
\end{enumerate}
For example, given $\omega=(0,1)$ and $T$ with $h\ge 1$  and
$\tau_k\notin \omega$ for all $k\in\Z$, the \mbox{B-s}plines with respect
to the knot sequence $T^\omega$
coincide with the Bernstein polynomials.
Analogous to the uniform case, the process of replacing a given knot
sequence $T$ by $T^\omega$ with maximal spacing $h^\omega\le|\omega|$
under the condition that knots within the domain $\omega$ remain
unchanged is called {\em non-uniform condensation}. It allows us to
carry over some notations and properties from~\cite{ReS}:
The {\em non-uniform condensed B-splines} with respect to $T^\omega$
are denoted by $b^\omega_i, i \in \Z$. The {\em set of indices of
active B-splines} is given by $I_\omega := 
\{i \in \Z : \supp b_i \cap \omega \neq \emptyset\}$, and we have
$I_\omega=\{i \in \Z : \supp b^\omega_i \cap \omega \neq \emptyset\}$.
The supports satisfy
$\supp b^\omega_i \subset \supp b_i$ for $	i \in I_\omega$
and the spanned spline spaces coincide, i.e.
$\spn\{b_i\chi(\omega) : i \in I_\omega\} =
	\spn\{b^\omega_i\chi(\omega) : i \in I_\omega\}$,
where $\chi(\omega)$ denotes the characteristic function of the domain.

The construction of a quasi-interpolant, which is necessary for
the proof of the main theorem, consists of two steps.
First, a projection into the space of polynomials and, second, a
representation of polynomials in terms of splines. We will briefly
recall some univariate results, which will be of use for this
construction.

Let $I^*\subseteq I$ be intervals in $\R$ and
$\{\ell_\alpha\in L^1(I) : \alpha < n\}$ the normalised Legendre
polynomials of degree $<n$ over $I$. A projection of functions
$f\in L^1(I^*)$ onto the space of polynomials is then given by
\[
\hat{f}=\sum_{\alpha=0}^{n-1}\langle f,\ell_\alpha\rangle \ell_\alpha.
\]
For estimates in the $L^p$-norm the following results will be useful.
\begin{lemma}\label{la:ell}
If the Legendre polynomials
$\{\ell_\alpha\in L^1(I):\alpha < n\}$ for $I\subset\R$ of size
$|I|=h$, $h>0$, are normalised such that
\[
	\int_I \ell_\alpha\ell_\beta
	= 
	\begin{cases}
		0 & \alpha\neq\beta\\
		1 & \alpha = \beta
	\end{cases},
\]
then
\begin{equation*}
\|\ell_\alpha\|_{I,p}\|\ell_\alpha\|_{I,\infty}
\le c_{n}\,h^{\frac{1}{p}-1}
\end{equation*}
with a constant only depending on $n$.
\end{lemma}
\begin{proof}
Without loss of generality, we assume that $I=[0,h]$. Let
$L_\alpha$, $\alpha<n$,
denote the normalised Legendre polynomials over $[-1,1]$. Then one
has $\ell_\alpha(x)=\sqrt{\frac{2}{h}}L_\alpha(\frac{2x-h}{h})$.
Since $L_\alpha$ is bounded by a constant only depending on $n$, we
get
\begin{equation*} 
\|\ell_\alpha\|_{I,\infty}\le c'_{n}\, h^{-\frac{1}{2}}.
\end{equation*}
Therefore, we have
\begin{equation*} 
	\|\ell_\alpha\|_{I,p}\|\ell_\alpha\|_{I,\infty}
	\le 
	\|1\|_{I,p}\|\ell_\alpha\|_{I,\infty}^2
	\le
	c_{n} h^{\frac{1}{p}-1}
\end{equation*}
with $c_{n}:=(c'_{n})^2$.
\end{proof}
The representation of polynomials in terms of non-uniform splines is
slightly different from the representation in the uniform case.
For a univariate B-spline $b_i$, let $I_i:=[\tau_j,\tau_{j+1})$ with
$j\in\{i,\dots,i+n-1\}$ and
$|I_i|\ge \frac{1}{n}|\supp b_i|$ be the longest knot interval of
$\supp b_i$.
For any polynomial $p \in \P_n$ of order $n$, the coefficient $s_i$
in the representation $p = \sum_i b_i s_i$ can be determined as
linear combination of the values of $p$ at the points 
$\lambda_{i,m}=\tau_j+\frac{m-1}{n-1}|I_i|$, $m=1,\dots,n$, and
certain weights ${w}^n_{i,1},\dots,{w}^n_{i,n}$ that can be
computed \cite{Hoellig2} by solving the linear system
$\sum_{m=1}^n {w}^n_{i,m}(m-l)^{n-1}=\frac{(n-1)^{n-1}}{|I_i|^{n-1}}
\psi_i^n(\lambda_{i,l})$, $l=1,\dots,n$, where
$\psi^n_i(\tau):=(\tau_{i+1}-\tau)\cdots(\tau_{i+n-1}-\tau)$. Since
$|\psi_i^n(\lambda_{i,l})|\le|\supp b_i|^{n-1}$, the right hand side is
bounded by $n^{2n}$. Additionally, the entries $(m-l)^{n-1}$ do not
depend on $T$. Therefore, the weights ${w}^n_{i,m}$ are bounded by
a constant only depending on $n$ and independent of $T$ and $b_i$.
Now, we have
\[
  p(x)= \sum_{i \in I_\omega}
  				b_i(x) \sum_{m=1}^n {w}^n_{i,m} p(\lambda_{i,m})
	,\quad
	x \in \omega
	.
\]
\section{Bivariate approximation}
Adapting univariate condensation to the bivariate case does not
solve the problem that the constant in \eqref{eq:stdest} depends
on the aspect ratio $\varrho:=\max_{i,j}h_i/h_j$. The dependence is
caused by B-splines
whose intersection with $\Omega$ consists of several connected
components (see \cite[Section 2.2]{ReS} for an example). This
problem can be solved by the concept of diversification, which will be
explained in this section and enables us to present the main
theorem.

At first we need some notations {and definitions recalling
relevant material from \cite{ReS}}.
\subsection{Notation}\label{notation}
Let $\Omega, M$ and $A$ denote subsets of $\R^2$ and $I$
an interval in $\R^2$ given by the Cartesian product of two
intervals in $\R$ with size $|I| := \sup\{x-y : x,y \in I\}$, where
the supremum is understood component-wise.
{The size of a subset $M$ is defined as the size of its bounding
box, i.e., $|M| := |\calB(M)|$ for {\em bounding box}
$\calB (M)=\calB_1(M)\times\calB_2(M):=\inf_{|I|}\{I:M\subset I\}$.}
Given $h \in [0,\infty]^2$ we define
{the {\em $h$-neighbourhood} of $M$ by}
$\calN(M,h):=\bigl\{x \in \R^2 : |\{x\} \cup M| \le |M|+h\bigr\}$.
The set of connected components of $M \cap \Omega$ is denoted by
$\calC_\Omega(M)$. If, in addition, $M$ is connected and
$M\subset\Omega$, we
{need two more notations, namely, for the elements of
$\calC_\Omega(\calB(M))$ and $\calC_\Omega(\calN(M,h))$ which
contain $M$, i.e., the {\em pruned bounding box}
$\calB_\Omega(M):=\{A\in\calC_\Omega(\calB(M)): M\subseteq A\}$
and the {\em local $h$-neighbourhood}
$\calN_\Omega(M,h):=\{A\in\calC_\Omega(\calN(M,h)):M\subseteq A\}.$}

From now on, the index $\sigma \in \{1,2\}$ is always addressing the two
coordinate directions, and the components of two-dimensional objects
are tagged by the prepending subscript $\sigma$. For instance,
$x = (x_1,x_2) \in \R^2$ or $x_j = (x_{1,j},x_{2,j})$ 
for a sequence $\{x_j\}_j$ in $\R^2$.

We give some further notations connected to the spline space $\SplSp$
spanned by the tensor product B-splines
$B_i(x) := b_{1,i_1}(x_1) b_{2,i_2}(x_2)$ of coordinate order
$n \in \N^2$ for $x \in \R^2|_\Omega$
and $i \in \Z^2$. Here, $b_{\sigma,i_\sigma}$ denotes the univariate
B-splines of order $n_\sigma$ with knots $T_\sigma$. We need
individual knots $\tau_k := (\tau_{1,k}, \tau_{2,k})$, $k\in\Z^2$,
{with} midpoints $\mu_k := (\tau_{k-(1,1)}+\tau_k)/2$
{and the {\em width of a cell} $h_k:=|\tau_k-\tau_{k-(1,1)}|$.
A {\em grid cell} $\Gamma_k$ is then defined by the
$h_k/2$-neighbourhood of $\mu_k$, that is,
$\Gamma_k:= \calN(\{\mu_k\},h_k/2)$.}
The support of a B-spline is denoted by
$S_i := \supp B_i = \calB(\{\tau_i\}\cup\{\tau_{i+n}\})$.
Further, we define $\bar n := \max(n_1,n_2)$.
\subsection{Main result}
\label{main}
As already mentioned in the introduction, diversification as well
as the set of domains are the same as in the uniform case. Therefore,
we just state the definitions and a few important comments
(see \cite{ReS} for {more} details) before we present the main result.
\begin{definition}
\label{def:omega}
Let $\Phi = [a,\varphi]$ be a pair consisting of the real number
$a >0$ and a continuous function $\varphi : X \to \R_{>0}$ defined on
the interval $X := [-a,a]$. With $X^\delta := [-a+\delta, a-\delta]$
the sub-interval with margin $\delta \in \R$, let
\[
  \Phi^\delta := \bigl\{x \in X^\delta \times \R : 
	\delta < x_2 < \varphi(x_1)\bigr\}
	.
\]
An {\em axis-aligned isometry} in $\R^2$ is a composition of a translation
and a rotation by an integer multiple of $\pi/2$.

A subset $\Omega\subset\R^2$ is called a 
{\em finite (or bounded) graph domain with parameter
$h_0 \in \R_{>0}$}
if there exists an finite index set $R \subset \N$, axis-aligned
isometries $\Sigma_r$, and pairs $\Phi_r = [a_r,\varphi_r]$ as above
with $a_r > h_0$ and $\min_{X_r} \varphi_r > h_0$ such that
\[
  \Omega = \bigcup_{r \in R} \Sigma_r(\Phi_r^\delta)
	,\quad
	0 \le \delta \le h_0
	.
\]
\end{definition}
{Basically, those domains consist of the union of subdomains
such that the boundary is piecewise represented by graphs of
continuous functions and $h_0$ guarantees a certain amount of
overlap of $\Sigma_r(\Phi_r^0)$. This overlap allows the following
argument, which}
will be used repeatedly, and without further
notice: Given $\Omega$ as above, let $M \subset \Omega$ be an
arbitrary subset of size $|M| \le 2h_0$. There exists a point
$x \in M$ satisfying $M \subset \calN_\Omega(\{x\},h_0)$ and an index
$r \in R$ such that $x \in \Sigma_r(\Phi_r^{h_0})$.
Hence, $M \subset \Sigma_r(\Phi_r^{0})$.
Without loss of generality, we can assume that $\Sigma_r$ is the
identity. Since the value of the index $r \in R$ will be irrelevant,
we drop the index $r$ and write $\Phi = [a,\varphi]$
instead of $\Phi_r = [a_r,\varphi_r]$ when examining $M$.
In particular, it is $M \subset \Phi^0$.

The idea of {\em diversification} is to use a copy of a given
B-spline $B_i$ for each connected component of its support in
$\Omega$, i.e., the {\em diversified B-splines} are defined by
\[
  B_j := B_i \chi(\gamma)
	,\quad
	j = (i,\gamma) \in J:= \bigl\{(i,\gamma) : i \in \Z^2,
  							\gamma \in \calC_\Omega(S_i)\bigr\}
	,
\]
with support $S_j := \supp B_j = \supp \chi(\gamma)$,
{see Figures~\ref{fig:1a} and~\ref{fig:1b}}.
From now on, the
subscripts $i \in \Z^2$ indicate standard B-splines, while subscripts
$j = (i,\gamma) \in J$ indicate their diversified descendants. The
diversified B-splines span the space $\tSplSp$ and we have
$\SplSp \subseteq \tSplSp$.

To formulate our main result, we define the
{\em anisotropic Sobolev space}
$W^n_p(\Omega)$, $1\le p\le\infty$ of order $n \in \N^2$ as the space
of real-valued functions $f\in L^p(\Omega)$ with partial derivatives
$\partial_1^{n_1}f, \partial_2^{n_2}f\in L^p(\Omega)$.
\begin{theorem}
\label{thm:approx}
Let $\Omega$ be a graph domain with parameter $h_0$, and let
$n\in\N^2$. There exists a constant $C$ depending only on $n$ and $p$
such that
	\[
		\inf_{s \in \tSplSp} \|f-s\|_{\Omega,p}
		\le C
		\left(h_1^{n_1}\|\partial_1^{n_1}f\|_{\Omega,p}
		     +h_2^{n_2}\|\partial_2^{n_2}f\|_{\Omega,p}
		\right)
	\]
for any $f\in W^{n}_p(\Omega)$, $1\le p\le \infty$, and any knot
sequence $T$ with grid width $h \le h_0/(\bar n+1)$.
\end{theorem}
\section{Proof}
The key element of the proof is the construction of a bounded
quasi-interpolant, which is given in Section~\ref{subsec:quasi}.
The necessary structure is provided by non-uniform condensation
as described in the first subsection. The method of transferring the
univariate condensation into two dimensions is originated in the
uniform condensation {described in \cite{ReS}}.
\subsection{Condensation}
\label{subsec:cond}
We describe how to apply condensation to diversified B-splines.
{The process of condensation is not applied directly with
respect to the complete $\Omega$. It is applied locally to subdomains
of $\Omega$ which are extensions of the supports $S_j$ and are determined by horizontally and vertically unbounded strips or
intervals depending on the supports $S_i$.}
Given the index $i \in \Z^2$, we define the
{unbounded extensions of the support $S_i$ in $x_1$- and
$x_2$-direction as}
the unbounded intervals $W_{1,i} := \calN(S_i,(\infty,0))$ and
$W_{2,i} := \calN(S_i,(0,\infty))$.
Further, for $j = (i,\gamma) \in J$ we denote
{the connected components of $W_{\sigma,i}$ containing the
support $S_j$ by
$W_{\sigma,j}:=\{A\in\calC_\Omega(W_{\sigma,i}) : S_j\subset A\}$.
See Figures~\ref{fig:1c} and~\ref{fig:1d} for examples of these sets.}
Now, condensation,
as described in Section~\ref{123}, is conducted with respect to the
univariate intervals $\omega_{\sigma,j} := \calB_\sigma(W_{\sigma,j})$
and we obtain the {\em non-uniform condensed diversified B-splines
(or briefly cdB-splines)}
\[
  B_j^*(x) := 
	b_{1,i_1}^{\omega_{1,j}}(x_1)
	b_{2,i_2}^{\omega_{2,j}}(x_2)
	\chi(\gamma)
	,\quad
	j = (i,\gamma) \in J
	.
\]
For their supports $S_j^* := \supp B_j^*$ we have
$S_j^* = S_j \subset S_i$ and $|S_j^*| = |S_j| \le \bar n h < h_0$.
{The supports of non-uniform cdB-splines are illustrated in
Figures~\ref{fig:1e} and~\ref{fig:1f}.}

We give some further notations related to the cdB-splines following
the notations of tensor product B-splines given in
Section~\ref{notation}. We use
$ T_j^*= (T_{1,j}^*,T_{2,j}^*)$ with
$T_{\sigma,j}^* = T_\sigma^{\omega_{\sigma,j}}$ and
maximal grid width $h^*_j$ for the condensed
knot sequences, $\tau_{j,k}^* = (\tau_{1,j,k}^*,\tau_{2,j,k}^*)$ with
$\tau_{\sigma,j,k}^*=\tau_{\sigma,k}^{\omega_{\sigma,j}}$ for the
individual knots, and
$\mu_{j,k}^* = (\tau_{j,k-(1,1)}^*+\tau_{j,k}^*)/2$ for
the midpoints. The local width of a cell is denoted by $h^*_{j,k}$.
\subsection{Local knot structure}
Locally, the condensed knot sequences $T^*_j$ still have a structure
such that some important properties of the spline space, like
partition of unity, remain unchanged. In this subsection, we
introduce some of those properties and the main arguments for their
validity.

{To analyse the structure of knot sequences in a proximity of
a grid cell $\Gamma_k$,} we denote the connected components of
$\Gamma_k\cap\Omega$ by $\Gamma_\ell$,
$\ell\in L := \bigl\{(k,\gamma) : k \in \Z^2,
	\gamma \in \calC_\Omega(\Gamma_k)\bigr\}$.
{We have $\Omega=\bigcup_{\ell\in L}\Gamma_\ell$.}	
Given $\ell\in L$, the neighbourhood
{of those cdB-splines which are} relevant for $\Gamma_\ell$ is
given by
\begin{equation*}
\Gamma^*_\ell := \bigcup_{j \in J_\ell} S_j^*\quad
\text{ with }
J_\ell := \{j \in J : S_j^* \cap \Gamma_\ell \neq \emptyset \}.
\end{equation*}
Since $|\Gamma^*_\ell| \le (2\bar n-1) h \le 2 h_0$, recalling
the arguments below Definition~\ref{def:omega}, we can assume that
$\Gamma^*_\ell \subset \Phi^0$.
{Therefore, for any $j \in J_\ell$, we have
$|\omega_{2,j}|\ge h_0$ and, thus, condensation leaves the knot
sequence in $x_2$-direction as it is and we get $T^*_{2,j} = T_2$.}
In $x_1$-direction there may be a modification which is the same for
all diversified B-splines with equal index $i_2$. More precisely,
for all
$j\in J_\ell^{i_2} := \bigl\{j' = (i',\gamma') \in J_\ell :
	i_2' = i_2\bigr\}$
there exists an interval $\omega'$ such that
we get $T_{1,j}^* := T_1^{\omega_{1,j}} = T_1^{\omega'}$ and
$b_{1,j} := b_{1,i_1}^{\omega_{1,j}} = b_{1,i_1}^{\omega'}$.
Those knot sequences depend only on the component $i_2$ of the index
$j = (i,\gamma) \in J_\ell^{i_2}$.
{For a proof, we refer to \cite{ReS}.}
Together, this allows us to write any spline $s \in \tSplSp$ locally
as
\begin{equation}
\label{eq:s}
  s(x) = \sum_{j \in J_\ell} s_j B^*_j(x)
	=
	\sum_{i_2\in\Z} b_{2,i_2}(x_2) \sum_{j \in J_\ell^{i_2}} 
	s_j b_{1,j}(x_1)
	,\quad
	x \in \Gamma_\ell
	.
\end{equation}
As mentioned above, due to the local structure, cdB-splines form
a partition of unity, i.e., if $s_j=1$ for all $j\in J_\ell$ in
\eqref{eq:s}, then
\begin{equation}
\label{eq:unity}
\sum_{j \in J_\ell} B^*_j(x)
	=
\sum_{i_2\in\Z} b_{2,i_2}(x_2) \sum_{j \in J_\ell^{i_2}} 
	b_{1,j}(x_1)
 = 
\sum_{i_2\in\Z} b_{2,i_2}(x_2) = 1
,\quad
x \in \Gamma_\ell
.
\end{equation}
Hence, using $|\Gamma_\ell|\le h_j^*$ for $j\in J_\ell$, it
follows that
\begin{equation}\label{eq:B_p}
\|B^*_j\|_{\Gamma_\ell,p}^p\le h_{1,j}^*h_{2,j}^*.
\end{equation}
\subsection{Representation of polynomials}
As already mentioned in the preceding section, the construction of a
quasi-interpolant for the proof consists of a projection onto the
space of polynomials combined with the representation of polynomials
in terms of splines, particularly cdB-splines. The latter will be
described in this section. 

We consider polynomials $p \in \P_n$ of coordinate order $n \in \N^2$.
The support of the unrestricted cdB-spline
{corresponding to} $B_j^*$ is given by
$S_j' := \calB(\{\tau^*_{j,i}\} \cup \{\tau^*_{j,i+n}\})$,
{see Figures~\ref{fig:1e} and~\ref{fig:1f}}.
Because of the tensor product structure, the
representation can be build analogous to the univariate case described
in Section~\ref{123}. We set
$\Gamma_{j,K}^*:=\calN(\{\mu_{j,K}^*\},h_{j,K}^*/2)$ {for
$K_\sigma\in\{i_\sigma+1,\dots,i_\sigma+n_\sigma\}$} with
{$|\Gamma_{j,K}^*|\ge \frac{1}{n}|S_j'|$} for the largest
subcell of $S_j'$. Then we define the linear functional
$P_j : \P_n \to \R$ by
\[
  P_j p := \sum_{m_1=1}^{n_1} \sum_{m_2=1}^{n_2}
	{w}_{j,m_1}^{n_1} {w}_{j,m_2}^{n_2}
	p(\lambda_{1,j,m}^*,\lambda_{2,j,m}^*)
	,\quad
	j = (i,\gamma)
	,
\]
with $\lambda_{\sigma,j,m}^*=\tau^*_{\sigma,j,K-(1,1)}+
\frac{m_\sigma-1}{n_\sigma-1}\cdot
|\tau^*_{\sigma,j,K-(1,1)}-\tau^*_{\sigma,j,K}|$.
We claim that 
\begin{equation}
\label{eq:p}
	p(x) = \sum_{j \in J} B_j^*(x) P_j p
	,\quad
	x \in \Omega
	.
\end{equation}
Since $T^*_{2,j}=T_2$, we have
$(\lambda_{2,j,m}^*)^{d_2} =\lambda_{2,i,m}^{d_2}$,
and ${w}_{j,m_2}^{n_2}= {w}_{i_2,m_2}^{n_2}$ is independent of
$j$. Using this, the proof of \eqref{eq:p} is
straightforward considering any monomial
$p^d(x) := x_1^{d_1} x_2^{{d_2}}$ of coordinate degree $d < n$ with
$x \in \Gamma_\ell$ for an arbitrary index $\ell \in L$.
\subsection{Quasi-interpolation}
\label{subsec:quasi}
In this section we construct a suitable quasi-interpolant. We will
use the representation of polynomials developed in the previous
subsection. Since the evaluation points $\lambda^*_{j,m}$ of $P_j$
may lie outside of $\Omega$, we first need a local approximation.

For any $j \in J$, we denote by $S_j^+ := \calN_\Omega(S_j^*,h_j^*)$
the local $h_j^*$-neighbourhood of $S_j^*$. Then, in the same way as in
the uniform case, there exists an interval $H_j^* \subset S_j^+$
of size $|H_j^*| = h_j^*$.
{The interval as well as the local $h_j^*$-neighbourhood are
illustrated in Figures~\ref{fig:1g}--\ref{fig:1j}.}
The interval is given by
$H_j^* := \omega' \times [\tau_{2,i}-h^*_{2,j},\tau_{2,i}]$,
where $\omega'\subset \omega_{1,j}$ with length $|h_{1,j}^*|$
containing some point $x = (x_1,\tau_{2,i}) \in S_j^*$ on the lower
bound of $S_j^*$. For more details see \cite{ReS}.
Next, we define the local approximation over $H^*_j$ as the linear
operator $A_j : W^n_p(H_j^*) \to \P_n$ defined by
\[
  A_j f := \sum_{\alpha<n}\langle f , \ell_{\alpha,j}
      \rangle \ell_{\alpha,j}
	,
\]
where $\{\ell_{\alpha,j} \in L^1(H^+_j) : \alpha < n\}$ are the
normalised tensor product Legendre polynomials over
$H^+_j:= \calN(H_j^*,\bar n h^*_j)$. So, $A_j$ is mapping the
function $f$ to the polynomial which is best approximating on the
interval $H_j^*$ with respect to the $L^2$-norm.

Now, we can define the {\em quasi-interpolant}
$Q : W^n_p(\Omega) \to \tSplSp$ by
\[
  Qf := \sum_{j \in J} B_j^* Q_j f
	,\quad
	Q_j := P_j A_j
	.
\]
The operator $Q$ reproduces polynomials $p\in \P_n$, what follows
immediately from \eqref{eq:p}. Further, the functionals
$Q_j:W^n_p(H_j^*) \rightarrow \R$, $j\in J$, can be written as
$Q_jf=\langle f,q_j\rangle$, which is the inner product of $f$ with
the polynomial
\begin{equation*}
q_j:=
\sum_{m_1=1}^{n_1} \sum_{m_2=1}^{n_2}\sum_{\alpha<n}
{w}_{j,m_1}^{n_1} {w}_{j,m_2}^{n_2}
\ell_{\alpha,j}(\lambda_{1,j,m}^*,\lambda_{2,j,m}^*)
\cdot
\ell_{\alpha,j}.
\end{equation*}
Since $\lambda_{j,m}^*\in S'_j\subset H^+_j$ and, as already remarked
in Section~\ref{123}, all $ w_{j,m_i}^{n_i}$, $i=1,2$, are
bounded by a constant only depending on $n$, we get
\begin{equation*}
\|q_j\|_{H_j^*,q}\le \tilde{c}_n \sum_{\alpha<n}
\|\ell_{\alpha,j}\|_{H^+_j,\infty}\|\ell_{\alpha,j}\|_{H^+_j,q}
\le c'_{n} (h_{1,j}^*h_{2,j}^*)^{\frac{1}{q}-1},
\end{equation*}
where we used Lemma~\ref{la:ell} and
$c'_{n}:=\tilde{c}_n\cdot c_{n}$. Therefore, using
H\"older's inequality, we see that the functionals $Q_j$ are
bounded by
\begin{equation}
\label{eq:Q_j}
|Q_jf|=|\langle f,q_j\rangle|
\le c'_n (h_{1,j}^*h_{2,j}^*)^{\frac{1}{q}-1}\|f\|_{H_j^*,p}.
\end{equation}
Hence, for $p=\infty$ and $q=1$ the boundedness of the spline $Qf$
on $\Gamma_\ell$ follows directly, since, using \eqref{eq:unity},
\begin{equation*}
\|Qf\|_{\Gamma_\ell,\infty}\le 
\max_{j\in J_\ell}|Q_jf|\le
c'_n \max_{j\in J_\ell}\|f\|_{H^*_j,\infty}.
\end{equation*}
Applying Minkowski's inequality and \eqref{eq:B_p} as well as
\eqref{eq:Q_j} in the case of $p<\infty$, we have
\begin{equation*}
\|Qf\|_{\Gamma_\ell,p} 
 \le \sum_{j\in J_\ell}\|B^*_jQ_jf\|_{\Gamma_\ell,p}
= \sum_{j\in J_\ell} |Q_jf|\|B^*_j\|_{\Gamma_\ell,p}
 \le  c'_n\sum_{j\in J_\ell}\|f\|_{H_j^*,p}
 \le  \bar{n}^{2}c'_n\max_{j\in J_\ell}\|f\|_{H_j^*,p},
\end{equation*}
where we use in the last inequality that the number of elements in
$J_\ell$ is $n_1\cdot n_2\le \bar{n}^2$.

In summary, a bound of the spline $Qf$ on $\Gamma_\ell$ for
$1\le p\le\infty$ is given by
\begin{equation}
\label{eq:|Q|_p}
\|Qf\|_{\Gamma_\ell,p}\le c_{n,p} \max_{j\in J_\ell}\|f\|_{H_j^*,p}
\end{equation}
with $c_{n,p}:=c'_n$ for $p=\infty$ and
$c_{n,p}:=\bar{n}^{2} c'_n$ else.
\subsection{Error estimate}
To prove Theorem~\ref{thm:approx} for $f \in W^n_p(\Omega)$, we first
consider the approximation $s := Qf$ on an arbitrary grid cell
$\Gamma_\ell, \ell \in L$. We define
\[
 \Gamma^+_\ell := \calB_\Omega\Bigl(\bigcup_{j \in J_\ell} S_j^+\Bigr).	
\]
It contains all parts of $\Omega$ with potential influence on
$s|_{\Gamma_\ell}$. In particular, $H_j^* \subset \Gamma_\ell^+$ for
all $j \in J_\ell$. We can assume that $\Gamma^+_\ell \in \Phi^0$ since
$|\Gamma^+_\ell| \le (2\bar n+1) h \le 2h_0$.

In
\cite{ReS} the construction of a set $\Gamma$ with
$\Gamma_\ell^+\subset \Gamma \subset \Phi^0$ and
$|\Gamma| \le (2\bar n+2) h$ is shown. Since this construction can be
transfered without any changes, according to \cite{Reif}, Theorem 2.5,
there exists a polynomial $p \in \P_n$ such that the error
$\Delta := f-p$ satisfies
\[
  \|\Delta\|_{\Gamma,p} \le c^*_{n,p} \bigl(
	h_1^{n_1} \|\partial_1^{n_1} f\|_{\Omega,p} +
	h_2^{n_2} \|\partial_2^{n_2} f\|_{\Omega,p} \bigr)
	,
\]
where the constant $c^*_{n,p}$ depends only on $n$ and $p$.
Eventually, we use reproduction of polynomials by the
quasi-interpolant $Q$ and \eqref{eq:|Q|_p} to obtain
\begin{align*}
\|f-Qf\|_{\Gamma_\ell,p} & \le
	\|\Delta\|_{\Gamma_\ell,p} + \|Q\Delta\|_{\Gamma_\ell,p}
	\le \|\Delta\|_{\Gamma_\ell,p}
	+ c_{n,p}\max_{j\in J_\ell}	\|\Delta\|_{H_j^*,p} \\
  & \le (1+c_{n,p})\|\Delta\|_{\Gamma,p} 
  	\le (1+c_{n}) c^*_{n,p}  \bigl(
	h_1^{n_1} \|\partial_1^{n_1} f\|_{\Omega,p} +
	h_2^{n_2} \|\partial_2^{n_2} f\|_{\Omega,p} \bigr)
	.
\end{align*}
The index $\ell \in L$ was chosen arbitrarily and, on the other hand,
$|\Gamma| \le (2\bar n+1) h$ for all $\ell \in L$.
Therefore, those restricted intervals can only overlap
$\tilde{c}\cdot\bar{n}^2$ times, and the claim of Theorem~\ref{thm:approx}
follows with 
\[
C := \begin{cases}
	(1+c'_{n}) c^*_{n,p} & \text{for } p=\infty,\\
	(\bar{n}^2+\bar{n}^4c'_{n})c^*_{n,p}\,\tilde{c} & \text{else}.
\end{cases}
\]
\section{Conclusion}
In this paper we generalised the ideas presented in \cite{ReS} to
arbitrary knot sequences and to error estimates with respect to
anisotropic Sobolev norms. In comparison to the uniform theory the
following conclusions can be made:
\begin{itemize}
\item In the bivariate case, diversification of tensor product
B-splines provides a construction of spline spaces which have optimal
approximation properties not only for uniform knot sequences but also
for non-uniform knot sequences.
\item Non-uniform condensation is used for the construction of bounded
quasi-in\-ter\-po\-lants for arbitrary tensor product spline spaces.
Since non-uniform condensation, as described in Section~\ref{123}, is
only depending on the size of the domain and is independent of the number
of knots inside of the domain, it is even easier than the condensation
described in \cite{ReS}.
\item In three or more variables, diversification can be used for
spline spaces with arbitrary knot sequences, but as already shown in
\cite{ReS} even for convex domains with smooth boundary the constant
of the error estimate can depend on the mesh ratio. In the case of
convex domains diversification leaves the spline space as it is and
(non-uniform) condensation is just a change of basis, therefore the dependence
remains unaffected.
\end{itemize}

An open question for future research is how the constants of error
estimates for spline approximation are connected to the geometry of
domains in higher dimensions.
\begin{figure}[ht]
\centering
\subfigure[Support of $B_i$\label{fig:1a}]{\includegraphics[width=6.5cm]{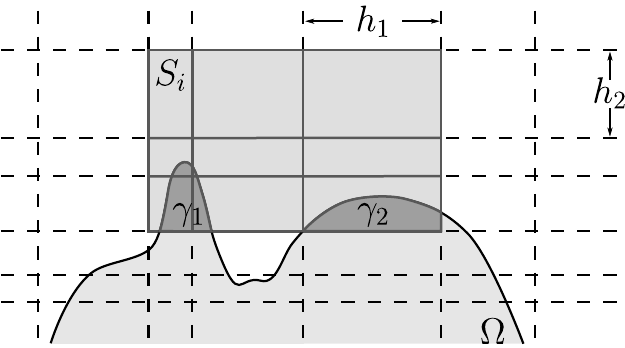}}\hspace{.5cm}
\subfigure[Supports of diversified B-splines\label{fig:1b}]{\includegraphics[width=6.5cm]{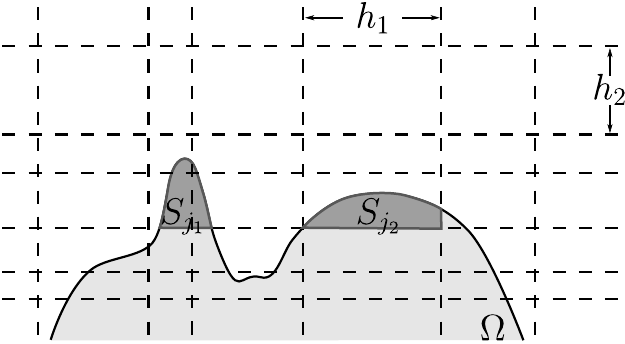}}\\

\subfigure[Extension of $S_i$ in $x_1$-direction\label{fig:1c}]{\includegraphics[width=6.5cm]{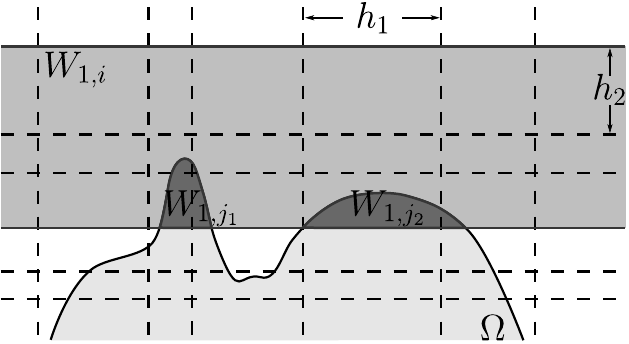}}\hspace{.5cm}
\subfigure[Extension of $S_i$ in $x_2$-direction\label{fig:1d}]{\includegraphics[width=6.5cm]{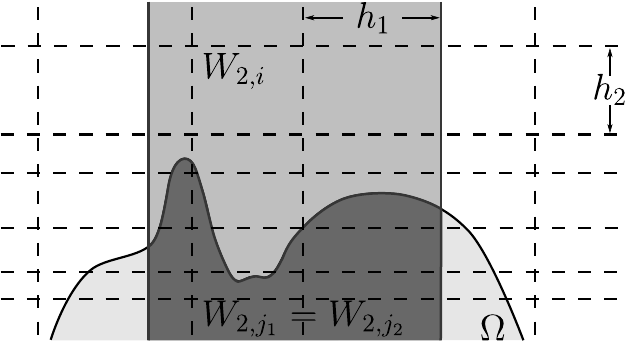}}\\

\subfigure[Support of (unrestricted) cdB-spline $B_{j_1}^*$\label{fig:1e}]{\includegraphics[width=6.5cm]{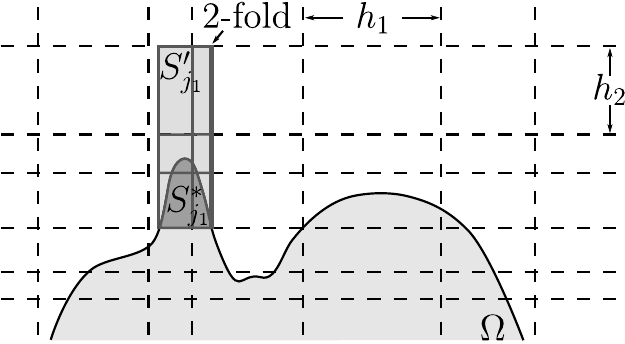}}\hspace{.5cm}
\subfigure[Support of (unrestricted) cdB-spline $B_{j_2}^*$\label{fig:1f}]{\includegraphics[width=6.5cm]{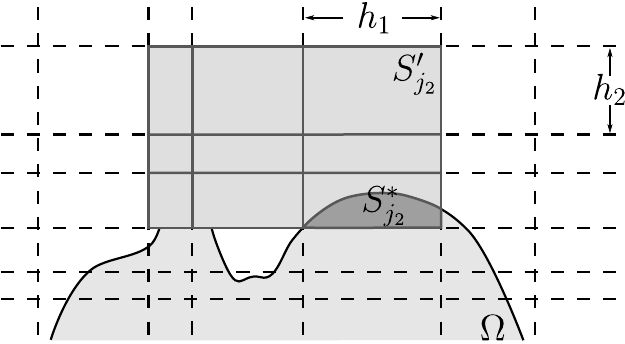}}\\

\subfigure[Neighbourhood of $S_{j_1}^*$\label{fig:1g}]{\includegraphics[width=6.5cm]{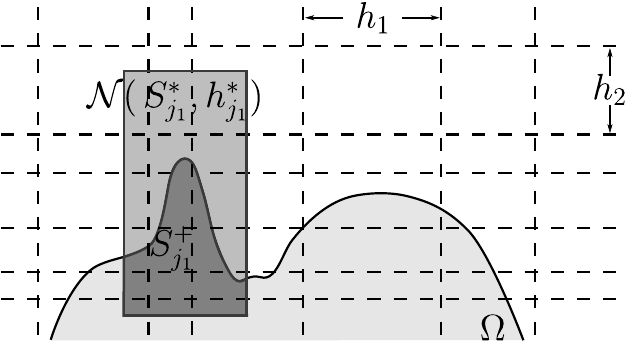}}\hspace{.5cm}
\subfigure[Neighbourhood of $S_{j_2}^*$\label{fig:1h}]{\includegraphics[width=6.5cm]{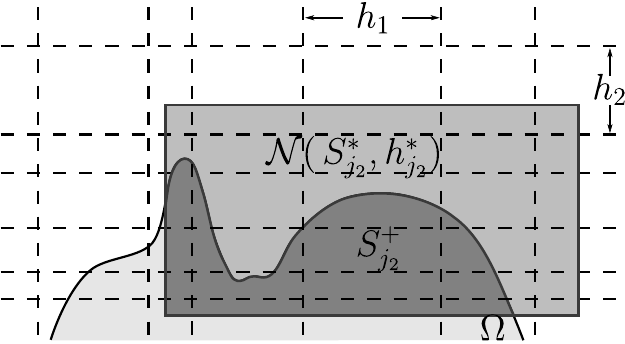}}\\

\subfigure[Domain $H_{j_1}^*$ for local approximation\label{fig:1i}]{\includegraphics[width=6.5cm]{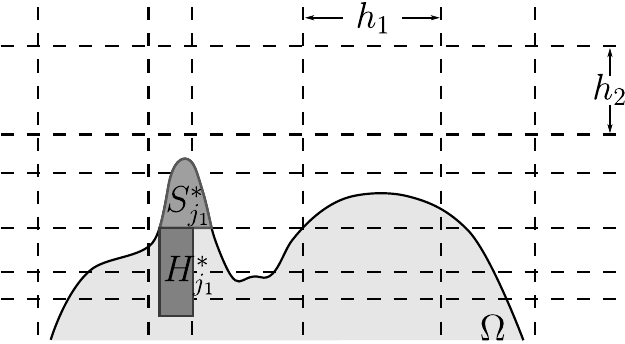}}\hspace{.5cm}
\subfigure[Domain $H_{j_2}^*$ for local approximation\label{fig:1j}]{\includegraphics[width=6.5cm]{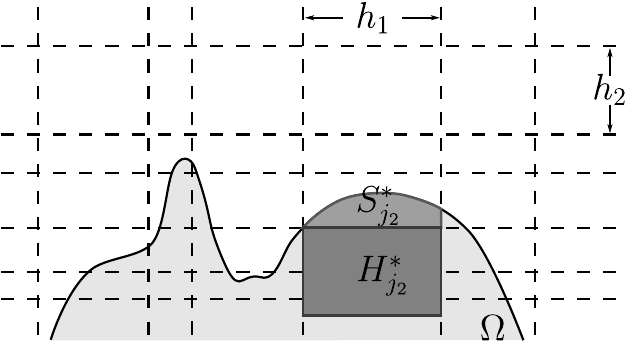}}
\caption{Examples for diversification, condensation, and related sets}
\label{fig:1}
\end{figure}
\bibliographystyle{alpha}
\bibliography{lit}
\end{document}